\documentclass[a4paper,autoref]{lipics-v2021}\nolinenumbers

\makeatother
\usepackage{amsmath,amssymb,amsfonts,latexsym}
\usepackage{color,graphicx}
\usepackage{url} 
\usepackage{fancybox}
\usepackage{algorithm}
\usepackage[noend]{algpseudocode}
\usepackage{mathtools}
\usepackage{enumerate,xspace}
\usepackage{thm-restate}
\usepackage{siunitx} 
\usepackage{tikz}
\usetikzlibrary{arrows.meta,calc,decorations.pathreplacing}
\hideLIPIcs

\newcommand{\eps}{\varepsilon}
\renewcommand{\epsilon}{\eps}

\theoremstyle{plain}

\newenvironment{myquote}%
  {\list{}{\leftmargin=4mm\rightmargin=4mm}\item[]}%
  {\endlist}

\renewcommand{\leq}{\leqslant}
\renewcommand{\geq}{\geqslant}

\newcommand{\opt}{\mbox{{\sc opt}}\xspace}

\renewcommand{\subjclass}[1]{\par\smallskip\noindent\textbf{Mathematics Subject Classification (2020).} #1\par\smallskip}

\definecolor{exactknavy}{RGB}{20,45,105}
\definecolor{navybar}{RGB}{24,52,93}
 \definecolor{tiletop}{RGB}{196,211,230}
 \definecolor{tilebottom}{RGB}{207,225,214}
\definecolor{tileleft}{RGB}{238,224,194}
 \definecolor{tileright}{RGB}{222,211,233}
\counterwithin{remark}{section}
\counterwithin{lemma}{section}
\counterwithin{theorem}{section}
\counterwithin{proposition}{section}
\counterwithin{claim}{section}
\counterwithin{observation}{section}
\counterwithin{corollary}{section}

\title{Exact-Distance Domination in Grid Graphs}

\author{Sandip Das}{Advanced Computing and Microelectronics Unit, Indian Statistical Institute, Kolkata, India}{sandipdas@isical.ac.in}{}{}
\author{Sweta Das}{Advanced Computing and Microelectronics Unit, Indian Statistical Institute, Kolkata, India}{swetamath98@gmail.com}{}{}
\author{Arpan Sadhukhan}{Department of Mathematics, Indian Institute of Technology, Dharwad, India}{ra.a.sadhukhan@iitdh.ac.in}{}{}

\authorrunning{S.~Das, S.~Das, A.~Sadhukhan} 
 \Copyright{Sandip Das, Sweta Das, Arpan Sadhukhan}


\makeatletter
\renewcommand{\subjclassHeading}{%
  \textcolor{lipicsGray}{\fontsize{9}{12}\sffamily\bfseries
    Mathematics Subject Classification (2020)\enskip}%
}

\gdef\@subjclass{ 05C12, 05C69}        
\gdef\@ccsdescString{\@subjclass}     

\providecommand{\ccsdesc}[2][]{}
\makeatother

\begin{document}
\setcounter{page}{0}
\maketitle

\begin{abstract}
Let $G_n$ be the $n\times n$ square grid, and let
$k\geq 2$. A set $D\subseteq V(G_n)$ is an
\emph{exact-distance $k$-dominating set} if every vertex
$v\in V(G_n)\setminus D$ has a vertex $u\in D$ with
$d(u,v)=k$. We write $D_{\mathrm{opt}}^{(k)}(G_n)$ for the minimum
cardinality of such a set. For every fixed $k$, consider the limit
$
\delta_k=
\lim_{n\to\infty}
\frac{D_{\mathrm{opt}}^{(k)}(G_n)}{n^2}.
$ We prove that, for every fixed \(k\geq 3\), $
\frac{1}{4k}
\leq
\delta_k
\leq
\frac{k-1}{3k^2-k-1}.
$
For $k=2$, the exact value $\delta_2=1/9$ follows directly.

\end{abstract}


\section{Introduction}

A \emph{dominating set} of a graph is a set of vertices such that every vertex
outside the set has a neighbor in it. Domination and its many variants form
a substantial area of graph theory; standard terminology and classical
results can be found in the monograph of Haynes, Hedetniemi, and
Slater~\cite{HaynesHedetniemiSlater1998}. Grid graphs are a particularly
natural setting for domination problems. Their regular local geometry makes
periodic constructions possible, while boundary effects and unavoidable
interactions between local coverage regions make sharp estimates difficult.
For ordinary domination, the domination number of all sufficiently large
rectangular grids was determined by Goncalves, Pinlou, Rao, and
Thomasse~\cite{GoncalvesPinlouRaoThomasse2011}.

The usual distance-\(k\) domination problem requires every vertex outside the
selected set to lie at a distance at most \(k\) from a selected vertex. Distance
domination and related broadcast parameters on grids have been studied using
geometric, algebraic, and algorithmic constructions; for example, see,
\cite{BlessingInskoJohnsonMauretour2015, FarinaGrez2016,FataSmithSundaram2013}.
In those settings, a selected vertex covers an entire metric ball. The problem considered here is different: the prescribed distance must be attained exactly.

Let $G_n$ be the $n \times n$ grid with $V(G_n)=\{1,2,\ldots,n\}^2$,
where the distance between two vertices is calculated as follows
\[
d\bigl((i,j),(i',j')\bigr)
=
|i-i'|+|j-j'|.
\]
For an integer \(k\geq 2\), a set \(D\subseteq V(G_n)\) is called an
\emph{exact-distance \(k\)-dominating set} if, for every
\(v\in V(G_n)\setminus D\), there exists \(u\in D\) such that $d(u,v)=k$.
We denote the minimum cardinality of such a set by $D_{\mathrm{opt}}^{(k)}(G_n)$.

When \(k=2\), this is the standard notion of hop domination: every vertex
outside the selected set has a selected vertex at a distance exactly
two~\cite{henning20172, NatarajanAyyaswamy2015}. We use the term
\emph{exact-distance \(k\)-domination} for arbitrary \(k\), since it states
the metric requirement is explicit and avoids confusion with ordinary
distance domination.

Our definition should also be distinguished from exact \(k\)-step
domination. In the latter notion, every vertex is required to have a unique
selected vertex at the prescribed distance; equivalently, the relevant
exact-distance neighborhoods form a partition of the vertex
set~\cite{ChartrandHararyHossainSchultz1995,Hersh1999}. In the present paper,
the distance requirement is imposed only on vertices outside \(D\), and
multiple coverage is allowed. For closely related papers we refer the readers to~\cite{anusha2024further, ayyaswamy2015bounds, ayyaswamy2018note,  packiavathihop, jalalvand2017complexity, fujita2025tight, henning2020algorithm, henning20172, Hersh1999,   karthika2025hop, karthika2025polynomial, shanmugavelan2021hop}.

The goal of the paper is to derive asymptotic upper and lower bounds for exact-distance $k$-domination in grids. In view of this, we define the \emph{exact-distance $k$-domination density} as \[
\delta_k
=
\lim_{n\to\infty}
\frac{D_{\mathrm{opt}}^{(k)}(G_n)}{n^2}.
\]

It is easy to show that the above limit exists (\ref{prop: limit exists}). 

\subparagraph*{Our contribution.} In Section~\ref{sec:lower bound}, we show that for every fixed \(k\geq3\), 
$
D_{\mathrm{opt}}^{(k)}(G_n)\geq \frac{n^2}{4k}-O_k(n)$.
 The proof exploits the geometry of the grid that forces many points to be dominated multiple times by any exact-distance $k$-dominating set. Combining these observations and using a careful charging argument, we derive the lower bound above. \\
Next, in Section~\ref{sec:upper bound}, by a periodic vertical-bar construction, for $k\geq 2$, we show that
\[
D_{\mathrm{opt}}^{(k)}(G_n)\leq \frac{k-1}{3k^2-k-1}n^2+O_k(n).
\]
The normalized quantity \(D_{\mathrm{opt}}^{(k)}(G_n)/n^2\) therefore converges to a limit \(\delta_k\) satisfying for all $k\geq 3$,
\[
\frac1{4k}\leq\delta_k\leq\frac{k-1}{3k^2-k-1}.
\]
For \(k=2\), we get asymptotically tight bound and prove that \(\delta_2=1/9\). \\ \\

\section{Preliminaries}

For a vertex $u\in V(G_n)$, define the exact $k$-sphere around $u$ by
\[
S_k(u)=\{x\in V(G_n):d(u,x)=k\}.
\]
Note that a $k$-sphere is not the whole ball of radius $k$. It is only the boundary layer at distance exactly $k$. So a selected vertex $u$ covers $u$ itself and the vertices exactly $k$ steps away, but it does not cover the vertices at distances $1,2,\ldots, k-1$. These interior vertices are what we informally call “holes” or "gaps" that the $k$-sphere leaves behind.
We also define the closed exact $k$-coverage set by $
T_k(u)=\{u\}\cup S_k(u)$.

\begin{figure}
\begin{center}
\includegraphics[width=0.8\columnwidth]{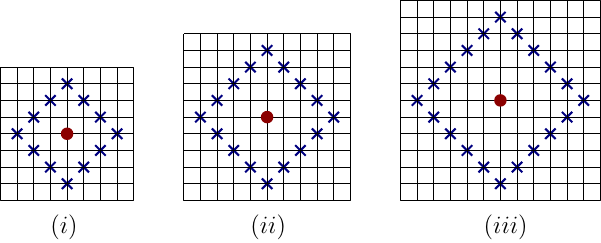}
\end{center}
\caption{The pictures (i), (ii), (iii) corresponds to exact $k$-spheres (set of blue points) around a point (red) for $k=3,4,5$ respectively.}
\label{fig:standardshape}
\end{figure}

The word ``closed'' means that we include the selected vertex $u$ itself. Let $D\subseteq V(G_n)$ be an exact $k$-dominating set. For every vertex $x\in V(G_n)$, define its \emph{charge} by
\[
\mathrm{ch}(x)=|\{u\in D:x\in T_k(u)\}|.
\]
This charge is exactly the coverage multiplicity of $x$: it counts how many selected vertices cover $x$.

With slight abuse of notation, in the rest of the paper, we sometimes denote an exact-distance $k$-dominating set by simply $k$-dominating set. Throughout the paper, \(O_k(n)\) denotes a quantity whose absolute value is
at most \(c_k n\), where the constant \(c_k\) may depend on \(k\), but not
on \(n\).

It is easy to observe that away from the boundary of the grid, the exact $k$-sphere has size $|S_k(u)|=4k$, because the integer solutions of $|a|+|b|=k$ form the boundary of a diamond and there are $4k$ such points. Hence $|T_k(u)|=4k+1$, for an interior vertex $u$. For points closer to boundary it is clear that the $k$-sphere will have even less cardinality. Thus we get a trivial lower bound on $D^{(k)}_{\opt}(G_n)$ which we note below.

\begin{observation}\label{obs:trivial lower bound}
    $D^{(k)}_{\opt}(G_n)\geq \frac{n^2}{4k+1}$
\end{observation}

Observe that $D^{(k)}_{\opt}(G_n)$ is trivially upper bounded by the size of the grid which is $n^2$. Next, we show in the following proposition below that the limit $\delta_k$ defined in the introduction always exists.

\begin{proposition}\label{prop: limit exists}

For every fixed \(k\geq2\), the limit $\delta_k$
exists. Moreover,
\[
\delta_k
=
\inf_{m\geq1}
\frac{D_{\mathrm{opt}}^{(k)}(G_m)}{m^2}.
\]

\end{proposition}

\begin{proof}
Consider, $ \gamma_k(n)=D_{\mathrm{opt}}^{(k)}(G_n)$. Fix \(m\geq 1\). Tiling a $ \left\lfloor\frac{n}{m}\right\rfloor m \times \left\lfloor\frac{n}{m}\right\rfloor m $ subgrid of \(G_n\) by disjoint \(m\times m\) grids, placing an optimal exact-distance \(k\)-dominating set in each block, and selecting all remaining vertices give  $\gamma_k(n) \leq \left\lfloor\frac{n}{m}\right\rfloor^2\gamma_k(m)+2mn$. 

Dividing by \(n^2\) and letting \(n\to\infty\), we obtain \[ \limsup_{n\to\infty}\frac{\gamma_k(n)}{n^2} \leq \frac{\gamma_k(m)}{m^2}. \] Since this holds for every \(m\), $\limsup_{n\to\infty}\frac{\gamma_k(n)}{n^2} \leq \inf_{m\geq1}\frac{\gamma_k(m)}{m^2}$.  The reverse inequality for the limit inferior follows immediately from the definition of the infimum. Hence the limit exists and equals the stated infimum. \end{proof}
\section{Lower bound for exact-distance \(k\)-domination on grid}\label{sec:lower bound}

 In this section, we show that the trivial lower bound~\ref{obs:trivial lower bound} does not match the asymptotic optimal solution for all $k\geq 3$. In other words, there is no clean way to get a tiling of the grid using the shapes as in Figure~\ref{fig:standardshape} that forces only negligible ($O(n)$) amount of overlap only. The grid geometry somewhat prevents this clean tiling. We establish the result through a thorough charging argument. This is the core contribution of the paper.

Let \(D\subseteq V(G_n)\) be a \(k\)-dominating set. For every vertex
\(x\in V(G_n)\), recall that the charge of \(x\) by
\(\mathrm{ch}(x)=|\{u\in D:x\in T_k(u)\}|\). Since \(D\) is an \(k\)-dominating set, every vertex has charge at least \(1\). Indeed, if
\(x\in D\), then \(x\in T_k(x)\), and if \(x\notin D\), then some \(u\in D\)
satisfies \(d(u,x)=k\), so \(x\in T_k(u)\).

Recall that the total excess charge by $E=\sum_{x\in V(G_n)}(\mathrm{ch}(x)-1)$.

A vertex of charge \(1\) contributes \(0\) to \(E\), a vertex of charge \(2\)
contributes \(1\), and so on. Since \(\sum_{x\in V(G_n)}\mathrm{ch}(x)=n^2+E\),
and every selected vertex contributes at most \(4k+1\) incidences to this sum,
we have
\[
 n^2+E\le (4k+1)|D|. \tag{1}
\]
The trivial lower bound uses only \(E\ge 0\). We improve it by proving that
\(E\) is large.

Let \(X=\{x\in V(G_n):\mathrm{ch}(x)\ge 2\}\) be the set of overcovered
vertices. For \(u\in D\), define the credit of \(u\) by
\(h(u)=|T_k(u)\cap X|\). Thus \(h(u)\) counts how many vertices in the coverage
set of \(u\) are overcovered. Finally, let \(H=\sum_{u\in D}h(u)\).

\begin{observation}\label{obs:H-sum}
\(H=\sum_{x\in X}\mathrm{ch}(x)\).
\end{observation}

\begin{proof}
A vertex \(x\in X\) is counted once in \(h(u)\) for each selected vertex
\(u\in D\) satisfying \(x\in T_k(u)\). The number of such selected vertices is
exactly \(\mathrm{ch}(x)\). Summing over all \(x\in X\) gives the identity.
\end{proof}

\begin{observation}\label{obs:E-H}
\(E\ge H/2\).
\end{observation}

\begin{proof}
For \(x\in X\), we have \(\mathrm{ch}(x)\ge 2\), and hence
\(\mathrm{ch}(x)\le 2(\mathrm{ch}(x)-1)\). By Observation~\ref{obs:H-sum},
\[
H=\sum_{x\in X}\mathrm{ch}(x)
\le 2\sum_{x\in X}(\mathrm{ch}(x)-1)
\le 2E.
\]
Thus \(E\ge H/2\).
\end{proof}

Next we show when the intersection of two exact \(k\)-spheres is non-empty.

\begin{lemma}\label{lem:sphere-intersection}
Let \(p,q\) be two vertices of the infinite square grid. Suppose \(d(p,q)=d\),
where \(d\) is even and \(2\le d\le 2k-2\). Then
\(|S_k(p)\cap S_k(q)|\ge 2\).
\end{lemma}

\begin{proof}
By translating, reflecting, and interchanging the two coordinates, we may
assume \(p=(0,0)\) and \(q=(a,b)\), where \(a\ge b\ge 0\) and \(a+b=d\). Since
\(d\) is even, \(a\) and \(b\) have the same parity. Hence \((a-b)/2\) is an
integer.

Define two lattice points
\[
 x_1=\left(\frac{a-b}{2},\ k-\frac{a-b}{2}\right),\qquad
 x_2=\left(\frac{a+b}{2},\ \frac{a+b}{2}-k\right).
\]
We claim that both points are at distance exactly \(k\) from both \(p\) and
\(q\). First, \(d(x_1,p)=k\). Also, since \(a+b\le 2k-2\), we have
\(k-(a+b)/2\ge 1\), and therefore
\[
 d(x_1,q)=\frac{a+b}{2}+\left(k-\frac{a+b}{2}\right)=k.
\]
Similarly, \(d(x_2,p)=k\). Moreover,
\[
 d(x_2,q)=\left|\frac{a+b}{2}-a\right|+
 \left|\frac{a+b}{2}-k-b\right|=\frac{a-b}{2}+
 \left(k-\frac{a-b}{2}\right)=k.
\]
Finally, \(x_1\ne x_2\), because the second coordinate of \(x_1\) is positive,
whereas the second coordinate of \(x_2\) is at most \(-1\). Hence the two exact
\(k\)-spheres have at least two common vertices.
\end{proof}

We now separate the odd and even cases. This is necessary because of parity.
When \(k\) is odd, a vertex at distance \(1\) from \(u\) can force a second
selected vertex whose center is at distance \(k-1\) or \(k+1\) from \(u\). When
\(k\) is even, a vertex at distance \(1\) has the wrong parity, so we use holes
at distance \(2\) instead.

\subsubsection*{Odd values of \(k\)}

Assume throughout this part that \(k\ge 3\) is odd. Define
\[
 I_o=\{(i,j)\in V(G_n):k+2\le i,j\le n-k-1\}.
\]
Thus \(|I_o|=(n-2k-2)^2\), provided \(n\ge 2k+3\). The reason for this choice
is that every vertex within distance at most \(k+1\) from a vertex of \(I_o\)
still lies inside the grid.

\begin{lemma}\label{lem:odd-holes-selected}
Let \(u\in D\cap I_o\). If \(h(u)\le 1\), then all four ordinary grid-neighbors
of \(u\) belong to \(D\).
\end{lemma}

\begin{proof}
Let \(w\) be an ordinary grid-neighbor of \(u\). Suppose, for contradiction,
that \(w\notin D\). Since \(D\) is a \(k\)-dominating set, there exists
\(z\in D\) such that \(d(z,w)=k\). Since \(d(u,w)=1\), the triangle inequality
gives \(k-1\le d(u,z)\le k+1\). Also, \(d(u,w)=1\) and \(d(w,z)=k\) are both odd,
so \(u\) and \(z\) have the same parity. Hence \(d(u,z)\) is even. Since \(k\)
is odd, this forces \(d(u,z)\in\{k-1,k+1\}\).

Both \(k-1\) and \(k+1\) lie between \(2\) and \(2k-2\). By Lemma~\ref{lem:sphere-intersection},
\(S_k(u)\cap S_k(z)\) contains at least two vertices. These common vertices lie
in \(S_k(u)\), and hence, inside \(G_n\), because \(u\in I_o\). They are covered
by both \(u\) and \(z\), so they are overcovered vertices inside \(T_k(u)\). This
implies \(h(u)\ge 2\), contradicting \(h(u)\le 1\). Therefore \(w\in D\).
\end{proof}

\begin{lemma}\label{lem:odd-extra-credit}
Let \(u\in D\cap I_o\) satisfy \(h(u)\le 1\). If \(v\) is one of the four
ordinary grid-neighbors of \(u\), then \(h(v)\ge 4\).
\end{lemma}

\begin{proof}
By Lemma~\ref{lem:odd-holes-selected}, all four ordinary grid-neighbors of
\(u\) lie in \(D\). By translation and symmetry, assume locally that
\(u=(0,0)\) and \(v=(1,0)\). Then \((0,1)\) and \((0,-1)\) also lie in \(D\).

The exact \(k\)-spheres around \((1,0)\) and \((0,1)\) have the two common
vertices \((1,k)\) and \((0,1-k)\). The exact \(k\)-spheres around \((1,0)\) and
\((0,-1)\) have the two common vertices \((1,-k)\) and \((0,k-1)\). These four
vertices are distinct, lie in \(T_k(v)\), and are overcovered. They lie inside
\(G_n\) because each is within distance at most \(k+1\) from \(u\), and
\(u\in I_o\). Hence, \(h(v)\ge 4\).
\end{proof}

\begin{lemma}\label{lem:odd-H}
If \(k\ge 3\) is odd, then \(H\ge 2|D\cap I_o|\).
\end{lemma}

\begin{proof}
We use credit redistribution. Think of each selected vertex \(v\in D\) as
carrying initial credit \(h(v)\). We redistribute credit so that every vertex of
\(D\cap I_o\) finishes with at least \(2\) units of credit.

If \(u\in D\cap I_o\) has \(h(u)\ge 2\), then it already has enough credit. If
\(h(u)\le 1\), then by Lemma~\ref{lem:odd-holes-selected}, all four ordinary
neighbors of \(u\) lie in \(D\). The deficit of \(u\) is at most \(2\), so \(u\)
takes this deficit equally from its four selected neighbors. Thus, it asks each
neighbor for at most \(1/2\) units of credit.

A fixed selected vertex can be an ordinary grid-neighbor of at most four
vertices \(u\), so it is asked to give away at most \(4\cdot(1/2)=2\) units of
credit. Moreover, every vertex that is asked to give credit has \(h\ge 4\), by
Lemma~\ref{lem:odd-extra-credit}. Thus, it can give away at most \(2\) units and
still keep at least \(2\) for itself, if it is also in \(D\cap I_o\). If it is
outside \(D\cap I_o\), it has no interior quota to satisfy. Therefore, the
redistribution is valid, and \(H\ge 2|D\cap I_o|\).
\end{proof}

\begin{theorem}\label{thm:odd-k-lower}
Let \(k\ge 3\) be odd. If \(n\ge 2k+3\), then 
\[D^{(k)}_{\mathrm{opt}}(G_n)\ge \frac{(n-2k-2)^2}{4k}\].

In particular, for fixed odd \(k\ge 3\),
\(D^{(k)}_{\mathrm{opt}}(G_n)\ge n^2/(4k)-O_k(n)\).
\end{theorem}

\begin{proof}
By Lemma~\ref{lem:odd-H} and Observation~\ref{obs:E-H}, we have
\(E\ge |D\cap I_o|\). Since \(|I_o|=(n-2k-2)^2\),
\(|D\cap I_o|\ge |D|-(n^2-|I_o|)\). Hence
\(E\ge |D|-n^2+|I_o|\). Substituting this into (1), we obtain
\[
 n^2+|D|-n^2+|I_o|\le (4k+1)|D|.
\]
Thus, \(|I_o|\le 4k|D|\), and therefore,
\(|D|\ge |I_o|/(4k)=(n-2k-2)^2/(4k)\). Since \(D\) was arbitrary, the theorem
follows.
\end{proof}

\subsubsection*{Even values of \(k\)}

Now assume \(k\ge 4\) is even. Define
\[
 I_e=\{(i,j)\in V(G_n):k+3\le i,j\le n-k-2\}.
\]
Thus \(|I_e|=(n-2k-4)^2\), provided \(n\ge 2k+5\). Here we use a slightly
larger buffer because the credit-forcing step will use vertices that may be at a
distance \(k+2\) from the central vertex.

\begin{lemma}\label{lem:even-holes-selected}
Let \(u\in D\cap I_e\). If \(h(u)\le 1\), then every vertex at distance \(2\)
from \(u\) belongs to \(D\).
\end{lemma}

\begin{proof}
Let \(w\) be a vertex with \(d(u,w)=2\). Suppose, for contradiction, that
\(w\notin D\). Since \(D\) is a \(k\)-dominating set, there exists
\(z\in D\) such that \(d(z,w)=k\). The triangle inequality gives
\(k-2\le d(u,z)\le k+2\). Since \(d(u,w)=2\) and \(d(w,z)=k\) are both even,
\(u\) and \(z\) have the same parity, so \(d(u,z)\) is even. Therefore
\(d(u,z)\in\{k-2,k,k+2\}\).

Because \(k\ge 4\), all three values \(k-2,k,k+2\) lie between \(2\) and
\(2k-2\). By Lemma~\ref{lem:sphere-intersection}, the exact \(k\)-spheres
around \(u\) and \(z\) have at least two common vertices. These vertices lie in
\(S_k(u)\), and hence inside the grid, because \(u\in I_e\). Therefore
\(T_k(u)\) contains at least two overcovered vertices, contradicting
\(h(u)\le 1\). Thus \(w\in D\).
\end{proof}

\begin{lemma}\label{lem:even-extra-credit}
Let \(u\in D\cap I_e\) satisfy \(h(u)\le 1\). If \(v\) is a vertex at distance
\(2\) from \(u\), then \(h(v)\ge 4\).
\end{lemma}

\begin{proof}
By Lemma~\ref{lem:even-holes-selected}, every vertex at distance \(2\) from
\(u\) lies in \(D\). By symmetry, it is enough to consider two cases.

First let \(u=(0,0)\) and \(v=(2,0)\). Then \((0,2)\) and \((0,-2)\) are also
selected. The exact \(k\)-spheres around \((2,0)\) and \((0,2)\) share
\((2,k)\) and \((0,2-k)\). The exact \(k\)-spheres around \((2,0)\) and
\((0,-2)\) share \((2,-k)\) and \((0,k-2)\). These four vertices are distinct,
overcovered, and lie in \(T_k(v)\).

Second let \(u=(0,0)\) and \(v=(1,1)\). Then \((1,-1)\) and \((-1,1)\) are also
selected. The exact \(k\)-spheres around \((1,1)\) and \((1,-1)\) share
\((k,0)\) and \((2-k,0)\). The exact \(k\)-spheres around \((1,1)\) and
\((-1,1)\) share \((0,k)\) and \((0,2-k)\). Again, these four vertices are
distinct, overcovered, and lie in \(T_k(v)\). All the displayed vertices lie
inside \(G_n\) because \(u\in I_e\). Hence \(h(v)\ge 4\) in all cases.
\end{proof}

\begin{lemma}\label{lem:even-H}
If \(k\ge 4\) is even, then \(H\ge 2|D\cap I_e|\).
\end{lemma}

\begin{proof}
The proof is the same credit redistribution argument, but using distance-\(2\)
neighbors instead of ordinary neighbors. If \(u\in D\cap I_e\) has \(h(u)\ge 2\),
there is nothing to do. If \(h(u)\le 1\), then Lemma~\ref{lem:even-holes-selected}
shows that all eight vertices at distance \(2\) from \(u\) lie in \(D\). The
deficit of \(u\) is at most \(2\), so \(u\) takes this deficit equally from
those eight selected vertices, asking each for at most \(1/4\) units of credit.

A fixed selected vertex can be at distance \(2\) from at most eight vertices
\(u\), so it is asked to give away at most \(8\cdot(1/4)=2\) units of credit.
Every vertex that is asked to give credit has \(h\ge 4\), by Lemma~\ref{lem:even-extra-credit}.
Therefore, it can give away at most \(2\) units and still keep \(2\) for itself,
if it is also in \(D\cap I_e\). Thus, the redistribution is valid, and
\(H\ge 2|D\cap I_e|\).
\end{proof}

\begin{theorem}\label{thm:even-k-lower}
Let \(k\ge 4\) be even. If \(n\ge 2k+5\), then
\[
 D^{(k)}_{\mathrm{opt}}(G_n)\ge \frac{(n-2k-4)^2}{4k}.
\]
In particular, for fixed even \(k\ge 4\),
\(D^{(k)}_{\mathrm{opt}}(G_n)\ge n^2/(4k)-O_k(n)\).
\end{theorem}

\begin{proof}
By Lemma~\ref{lem:even-H} and Observation~\ref{obs:E-H}, we have
\(E\ge |D\cap I_e|\). Since \(|I_e|=(n-2k-4)^2\), we get
\(E\ge |D|-n^2+|I_e|\). Substituting into (1) yields
\[
 n^2+|D|-n^2+|I_e|\le (4k+1)|D|.
\]
Thus \(|I_e|\le 4k|D|\), and hence
\(|D|\ge |I_e|/(4k)=(n-2k-4)^2/(4k)\). Since \(D\) was arbitrary, the theorem
follows.
\end{proof}

Combining Theorems~\ref{thm:odd-k-lower} and~\ref{thm:even-k-lower}, we obtain
the compact asymptotic statement
\[
 D^{(k)}_{\mathrm{opt}}(G_n)\ge \frac{1}{4k}n^2-O_k(n)
\]
for every fixed \(k\ge 3\).

\section{Upper bound for exact-distance \(k\)-domination on grid} \label{sec:upper bound}

Throughout this section, \(k\geq 2\) is fixed. We work first in the infinite
square lattice \(\mathbb{Z}^2\), endowed with the Manhattan distance
\[
d\bigl((x,y),(x',y')\bigr)=|x-x'|+|y-y'|.
\]
For \(u\in\mathbb{Z}^2\), write
\[
T_k(u)=\{u\}\cup\{z\in\mathbb{Z}^2:d(u,z)=k\}
\]
for its closed exact-distance \(k\)-coverage set.

\textbf{Idea overview:} First, we choose a vertical bar of selected
vertices and determine its coverage region exactly. We then label that region by a complete set of residues. Next we take copies of the vertical bar periodically (using modular arithmetic) and show that it dominates the infinite lattice. Refer to Figure~\ref{fig:exact-k-vertical-bar-construction}. Finally, we
add a width-\(k\) boundary strip and average over all residue shifts to obtain
a finite-grid bound.

\subsection{The vertical bar and its coverage region}

Set $L=2k-2$ and define the vertical bar
\[
B_k=\{(0,j):0\leq j\leq L-1\}.
\]
Thus, \(B_k\) contains exactly \(L=2k-2\) selected vertices. Let
\[
C_k=\bigcup_{q\in B_k}T_k(q)
\]
be the set covered by the whole bar. For an integer \(y\), define
\[
C_k(y)=\{x\in\mathbb{Z}:(x,y)\in C_k\}.
\]

The following proposition identifies the point set dominated by the vertical bar $B_k$. The proof is somewhat trivial, but we still give it for completeness.

\begin{proposition}
\label{prop:exact-k-upper-row-structure}
The nonempty rows of \(C_k\) are precisely the rows
\[
-k\leq y\leq L+k-1.
\]
More explicitly,

\begin{equation}
C_k(-k+t)=\{-t,-t+1,\ldots,t\}
\qquad (0\leq t\leq k-1),
\label{eq:exact-k-upper-lower-row}
\end{equation}
\begin{equation}
C_k(y)=\{-k,-k+1,\ldots,k\}
\qquad (0\leq y\leq L-1),
\label{eq:exact-k-upper-middle-row}
\end{equation}
and
\begin{equation}
C_k(L+t)
=
\{-(k-1-t),\ldots,k-1-t\}
\qquad (0\leq t\leq k-1).
\label{eq:exact-k-upper-upper-row}
\end{equation}

Consequently, the row widths are
\[
1,3,5,\ldots,2k-1,
\underbrace{2k+1,\ldots,2k+1}_{2k-2\text{ rows}},
2k-1,\ldots,5,3,1.
\]
\end{proposition}

\begin{proof}
Consider first a lower-tail row \(y=-k+t\), where
\(0\leq t\leq k-1\). For a bar vertex \((0,j)\), $|y-j|=k-t+j$.
A point \((x,y)\) lies at distance exactly \(k\) from \((0,j)\) precisely when $|x|+k-t+j=k$, or equivalently $|x|=t-j$.

This is possible exactly for \(0\leq j\leq t\). As \(j\) runs from \(0\)
to \(t\), the possible values of \(|x|\) are
\(t,t-1,\ldots,0\). Hence the whole interval
\(\{-t,-t+1,\ldots,t\}\) occurs, proving~\eqref{eq:exact-k-upper-lower-row}.

Now, let \(0\leq y\leq L-1\). Every covered point on this row satisfies
\(|x|\leq k\). Conversely, \((0,y)\) belongs to the selected bar. Let
\(1\leq |x|\leq k\), and put $d_0=k-|x|$.
Then \(0\leq d_0\leq k-1\). Since $L-1=2k-3$, at least one of the two numbers \(y\) and \(L-1-y\) is at least
\(k-1\). Therefore, there is a bar index \(j\) satisfying $|j-y|=d_0$.

For this \(j\), $d\bigl((x,y),(0,j)\bigr)
=|x|+d_0
=k$.
Thus, every \(x\in\{-k,-k+1,\ldots,k\}\) occurs, proving~\eqref{eq:exact-k-upper-middle-row}.

Finally, consider an upper-tail row \(y=L+t\), where
\(0\leq t\leq k-1\). Write a bar index in the form $j=L-1-h,
\qquad
0\leq h\leq L-1.
$
Then $|y-j|=t+1+h$.
The exact-distance equation becomes $|x|+t+1+h=k$, so $|x|=k-1-t-h$.
As \(h\) ranges from \(0\) to \(k-1-t\), the possible values of
\(|x|\) are \(k-1-t,k-2-t,\ldots,0\). This gives exactly the interval in~\eqref{eq:exact-k-upper-upper-row}. Rows outside the displayed range are farther than \(k\) from
every bar vertex and hence are empty.

\end{proof}

\begin{figure}[t]
\centering
\begin{tikzpicture}[x=0.48cm,y=0.48cm]
  \foreach \y/\a/\b/\w in {
    -4/0/0/1,
    -3/-1/1/3,
    -2/-2/2/5,
    -1/-3/3/7,
     0/-4/4/9,
     1/-4/4/9,
     2/-4/4/9,
     3/-4/4/9,
     4/-4/4/9,
     5/-4/4/9,
     6/-3/3/7,
     7/-2/2/5,
     8/-1/1/3,
     9/0/0/1}
  {
    \foreach \x in {\a,...,\b}
      \fill[black!32] (\x,\y) circle (1.35pt);
    \node[anchor=east,font=\scriptsize] at (-5.0,\y) {\(\w\)};
  }
  \draw[black!35] (-4.55,-4.55) rectangle (4.55,9.55);
  \foreach \y in {0,...,5}
    \filldraw[fill=exactknavy,draw=white,line width=0.35pt]
      (0,\y) circle (3.15pt);
  \draw[exactknavy,line width=1pt,rounded corners=2pt]
      (-0.43,-0.43) rectangle (0.43,5.43);
  \node[font=\scriptsize,rotate=90] at (5.15,-2.0) {lower tail};
  \node[font=\scriptsize,rotate=90] at (5.15,2.5) {middle rows};
  \node[font=\scriptsize,rotate=90] at (5.15,7.5) {upper tail};
  \node[font=\scriptsize,anchor=east] at (-5.0,10.25) {row width};
\end{tikzpicture}
\caption{The coverage region \(C_4\). The six navy vertices form the
bar \(B_4\). The gray points are all vertices covered by the bar. The row
widths are \(1,3,5,7\), then six rows of width \(9\), and finally
\(7,5,3,1\).}
\label{fig:exact-k-upper-model-bar}
\end{figure}
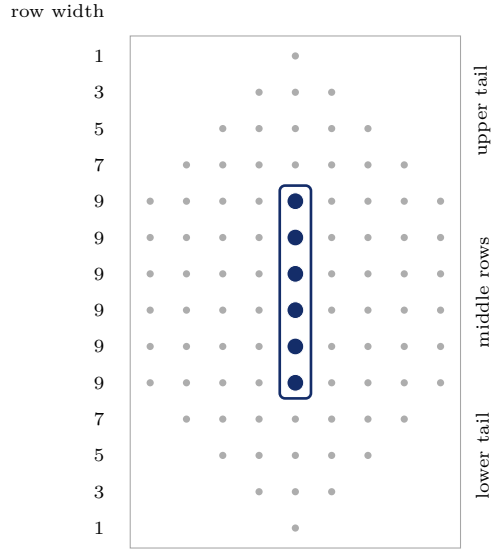

In the corollary below, we note the size of the covered region.
\begin{corollary}
\label{cor:exact-k-upper-model-size}
The number of vertices in \(C_k\) is
\[
|C_k|=2(3k^2-k-1).
\]
\end{corollary}

\begin{proof}
The \(L=2k-2\) middle rows each have width \(2k+1\). Each tail has row
widths \(1,3,\ldots,2k-1\), whose sum is \(k^2\). Therefore
\[
\begin{aligned}
|C_k|
&=(2k-2)(2k+1)+2(1+3+\cdots+(2k-1))\\
&=2(3k^2-k-1).
\end{aligned}
\]
\end{proof}

\subsection{Reading the coverage region using residues}

Set $s=2k+1$ and $M=2(3k^2-k-1)=|C_k|$.
The number \(s\) is the width of every middle row, while \(M\) is the total
number of vertices in the model region. Define 
\[
\phi_k(x,y)=x+sy\pmod M.
\]
Moving one lattice step to the right increases the label ($\phi$) by \(1\), and moving
one row upward increases it by exactly the middle-row width \(s\). 

In this subsection, we show that the linear labelling
\(\phi_k(x,y)=x+(2k+1)y\pmod M\) reads the entire coverage region \(C_k\)
as one complete system of residues modulo \(M\). The middle rows contribute
one consecutive cyclic interval, while the lower and upper tails interlace
without gaps or overlaps to fill exactly the remaining residues. This
complete-residue property which we prove in the lemma below is the key ingredient that allows the local
vertical-bar pattern to be repeated periodically over the whole lattice.

\begin{lemma}
\label{lem:exact-k-upper-complete-residues}
The restriction of \(\phi_k\) to \(C_k\) is a bijection
\[
\phi_k:C_k\longrightarrow\mathbb{Z}/M\mathbb{Z}.
\]
Equivalently, the \(M\) vertices of \(C_k\) receive all \(M\) residue
classes exactly once.
\end{lemma}

\begin{proof}
Put $A=4k^2-3k-3$.
We first keep the labels as ordinary integers and reduce modulo \(M\) only
at the end.

For a middle row \(0\leq y\leq L-1=2k-3\), Proposition~\ref{prop:exact-k-upper-row-structure}
gives \(-k\leq x\leq k\). Hence its labels form the interval $[sy-k,\,sy+k]$.
The next middle row starts at $s(y+1)-k=sy+k+1$,
which is exactly one more than the preceding endpoint. Thus all middle rows concatenate into $[-k,A]$.
Modulo \(M\), they occupy
\begin{equation}
\{M-k,M-k+1,\ldots,M-1\}
\cup
\{0,1,\ldots,A\}.
\label{eq:exact-k-upper-middle-residues}
\end{equation}
The missing residues are therefore the single interval
\begin{equation}
A+1,A+2,\ldots,M-k-1.
\label{eq:exact-k-upper-missing-residues}
\end{equation}

Now write a lower-tail row as \(y=-k+t\), with
\(0\leq t\leq k-1\). Its \(x\)-interval is \([-t,t]\), so its integer
labels form $[-sk+2kt,\,-sk+(2k+2)t]$.
Since $M-sk=A+1$,
adding \(M\) turns this into the congruent interval
\begin{equation}
I_t^-=[A+1+2kt,\,A+1+(2k+2)t].
\label{eq:exact-k-upper-lower-tail-interval}
\end{equation}

Similarly, write an upper-tail row as \(y=L+t\), with
\(0\leq t\leq k-1\). Its \(x\)-interval is
\([-(k-1-t),k-1-t]\), and a direct simplification gives the label interval
\begin{equation}
I_t^+=[A+2+(2k+2)t,\,A+2k(t+1)].
\label{eq:exact-k-upper-upper-tail-interval}
\end{equation}

The intervals in~\eqref{eq:exact-k-upper-lower-tail-interval} and
\eqref{eq:exact-k-upper-upper-tail-interval} fit together without gaps. Indeed, $\min I_t^+=\max I_t^-+1$,
and, for \(0\leq t<k-1\), $\min I_{t+1}^-=\max I_t^++1$.
Hence the tail intervals occur consecutively in the order
\[
I_0^-,I_0^+,I_1^-,I_1^+,\ldots,I_{k-1}^-,I_{k-1}^+.
\]
The first interval starts at \(A+1\), and the last one ends at $A+2k^2=M-k-1$.
Thus the tails cover exactly the missing interval
\eqref{eq:exact-k-upper-missing-residues}, once each. Together with
\eqref{eq:exact-k-upper-middle-residues}, this gives every residue modulo \(M\). Since
\(|C_k|=M\), no residue can occur twice.
\end{proof}

\begin{remark}
The complete-residue property is a special feature of the geometry of
\(C_k\), rather than an automatic consequence of its cardinality. In
general, a finite lattice region of size \(M\) need not admit any linear
labeling \(ax+by\pmod M\) that is injective on the region. Here the constant
middle-row width \(2k+1\), together with the precisely matched triangular
tails, makes the residue intervals concatenate without gaps or overlaps.
\end{remark}

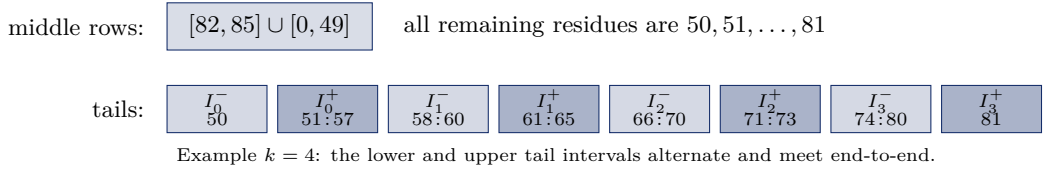
\begin{figure}[t]
\centering
\begin{tikzpicture}[x=1.18cm,y=0.78cm]
  \node[anchor=east,font=\small] at (-0.15,2.15) {middle rows:};
  \draw[fill=exactknavy!16,draw=exactknavy]
      (0,1.75) rectangle (2.3,2.55);
  \node[font=\small] at (1.15,2.15) {\([82,85]\cup[0,49]\)};
  \node[anchor=east,font=\small] at (-0.15,0.75) {tails:};
  \foreach \i/\lab/\rng/\shade in {
    0/{I_0^-}/{50}/18,
    1/{I_0^+}/{51\!:\!57}/36,
    2/{I_1^-}/{58\!:\!60}/18,
    3/{I_1^+}/{61\!:\!65}/36,
    4/{I_2^-}/{66\!:\!70}/18,
    5/{I_2^+}/{71\!:\!73}/36,
    6/{I_3^-}/{74\!:\!80}/18,
    7/{I_3^+}/{81}/36}
  {
    \draw[fill=exactknavy!\shade,draw=exactknavy]
      (1.24*\i,0.35) rectangle (1.24*\i+1.12,1.15);
    \node[font=\scriptsize] at (1.24*\i+0.56,0.88) {\(\lab\)};
    \node[font=\scriptsize] at (1.24*\i+0.56,0.58) {\(\rng\)};
  }
  \node[font=\small,anchor=west] at (2.55,2.15)
      {all remaining residues are \(50,51,\ldots,81\)};
  \node[font=\scriptsize,anchor=west] at (0,-0.05)
      {Example \(k=4\): the lower and upper tail intervals alternate and meet end-to-end.};
\end{tikzpicture}
\caption{The residue intervals for \(k=4\), where \(M=86\) and
\(A=49\). The middle rows cover \([82,85]\cup[0,49]\); the alternating
tail intervals cover the complementary interval \([50,81]\) with no gap
and no overlap.}
\label{fig:exact-k-upper-residue-intervals}
\end{figure}

\subsection{Dominating the infinite lattice using periodic placement of the vertical bar}

The labels of the selected vertices of the model bar are
\[
R_k
=
\{0,s,2s,\ldots,(L-1)s\}
\subseteq
\mathbb{Z}/M\mathbb{Z}.
\]
For a shift \(t\in\mathbb{Z}/M\mathbb{Z}\), define the periodic set
\[
D_\infty(t)
=
\{z\in\mathbb{Z}^2:\phi_k(z)\in t+R_k\}.
\]

\begin{proposition}
\label{prop:exact-k-upper-infinite}
For every shift \(t\), the set \(D_\infty(t)\) is an exact-distance
\(k\)-dominating set of \(\mathbb{Z}^2\).
\end{proposition}

\begin{proof}
Let \(v\in\mathbb{Z}^2\) be arbitrary. By
Lemma~\ref{lem:exact-k-upper-complete-residues}, there is a unique point
\(p\in C_k\) satisfying
\begin{equation}
\phi_k(p)=\phi_k(v)-t
\pmod M.
\label{eq:exact-k-upper-choose-p}
\end{equation}
Because \(p\in C_k\), some bar vertex \(q\in B_k\) satisfies either $p=q$ or $d(p,q)=k$. Set $a=v-p$ and $w=q+a$. Thus the model picture has been translated by the vector \(a\), sending
\(p\) to \(v\) and \(q\) to \(w\). The map \(\phi_k\) is additive under translations. From~\eqref{eq:exact-k-upper-choose-p},
\[
\phi_k(a)
=
\phi_k(v)-\phi_k(p)
=t
\pmod M.
\]
Therefore, $\phi_k(w)
=
\phi_k(q)+\phi_k(a)
=
\phi_k(q)+t
\in
t+R_k.$
Hence \(w\in D_\infty(t)\).

If \(p=q\), then \(w=v\), so \(v\) itself is selected. If \(p\neq q\),
translation preserves distance and therefore $d(v,w)=d(p,q)=k$.
Thus every lattice vertex is either selected or has a selected vertex at
distance exactly \(k\).
\end{proof}

\begin{figure}[t]
\centering
\begin{tikzpicture}[
  x=0.39cm,y=0.39cm,
  >=Latex,
  line cap=round,
  line join=round,
  every node/.style={font=\small}
]

\def\K{4}
\pgfmathtruncatemacro{\BarSize}{2*\K-2}
\pgfmathtruncatemacro{\BarTop}{2*\K-3}
\pgfmathtruncatemacro{\TileTop}{3*\K-3}
\pgfmathtruncatemacro{\ux}{2*\K+1}
\pgfmathtruncatemacro{\vy}{3*\K-2}
\pgfmathtruncatemacro{\wx}{\K+1}
\pgfmathtruncatemacro{\wy}{3*\K-3}
\pgfmathtruncatemacro{\xmin}{-3*\K-3}
\pgfmathtruncatemacro{\xmax}{ 3*\K+3}
\pgfmathtruncatemacro{\ymin}{-3*\K-2}
\pgfmathtruncatemacro{\ymax}{ 4*\K+1}

\draw[gray!22,step=1] (\xmin,\ymin) grid (\xmax,\ymax);

\newcommand{\ExactKBar}[3]{%
  \draw[#3,line width=0.55pt] (#1,#2) -- (#1,{#2+\BarTop});
  \foreach \j in {0,...,\BarTop}{
    \filldraw[fill=#3,draw=white,line width=0.35pt]
      (#1,{#2+\j}) circle[radius=0.18];
  }
}

\ExactKBar{\ux}{-1}{navybar!55}
\ExactKBar{-\ux}{1}{navybar!55}
\ExactKBar{-\K}{\vy}{navybar!55}
\ExactKBar{\K}{-\vy}{navybar!55}
\ExactKBar{\wx}{\wy}{navybar!55}
\ExactKBar{-\wx}{-\wy}{navybar!55}

\filldraw[fill=tilebottom,draw=navybar,line width=0.8pt]
  (0,-\K) -- (\K,0) -- (-\K,0) -- cycle;
\filldraw[fill=tileleft,draw=navybar,line width=0.8pt]
  (-\K,0) -- (0,0) -- (0,\BarTop) -- (-\K,\BarTop) -- cycle;
\filldraw[fill=tileright,draw=navybar,line width=0.8pt]
  (0,0) -- (\K,0) -- (\K,\BarTop) -- (0,\BarTop) -- cycle;
\filldraw[fill=tiletop,draw=navybar,line width=0.8pt]
  (-\K,\BarTop) -- (\K,\BarTop) -- (0,\TileTop) -- cycle;

\draw[navybar,line width=1.25pt]
  (0,-\K) -- (\K,0) -- (\K,\BarTop) --
  (0,\TileTop) -- (-\K,\BarTop) -- (-\K,0) -- cycle;

\ExactKBar{0}{0}{navybar}
\node[navybar,fill=white,inner sep=1.5pt]
  at (0,{\BarTop/2}) {$B_k$};

\draw[<->,line width=0.7pt]
  (-\K+0.25,2.1) -- (-0.35,2.1)
  node[midway,fill=white,inner sep=1pt] {$k$};
\draw[<->,line width=0.7pt]
  (0.35,2.1) -- (\K-0.25,2.1)
  node[midway,fill=white,inner sep=1pt] {$k$};

\draw[<->,line width=0.7pt]
  (0.85,-\K+0.2) -- (0.85,-0.2)
  node[midway,right,fill=white,inner sep=1pt] {$k$};
\draw[<->,line width=0.7pt]
  (0.85,\BarTop+0.2) -- (0.85,\TileTop-0.2)
  node[midway,right,fill=white,inner sep=1pt] {$k$};

\draw[decorate,decoration={brace,amplitude=4pt},line width=0.7pt]
  (-\K-0.55,0) -- (-\K-0.55,\BarTop)
  node[midway,left=5pt,align=right]
  {$2k-2$\\[-1pt]{\scriptsize selected vertices}};

\draw[decorate,decoration={brace,mirror,amplitude=4pt},line width=0.7pt]
  (\K+0.55,-\K) -- (\K+0.55,\TileTop)
  node[midway,right=5pt,align=left]
  {$4k-3$\\[-1pt]{\scriptsize lattice steps}};

\node[fill=white,inner sep=1.5pt]
  at (0,{\TileTop-1.0})
  {$C_k$, \quad $|C_k|=2(3k^2-k-1)$};

\node[navybar!80,fill=white,inner sep=1.5pt,anchor=west]
  (ulabel) at ({\ux+0.5},{-1+\BarTop+1.0})
  {$\mathbf u=(2k+1,-1)$};
\draw[->,navybar!80,line width=0.7pt]
  (ulabel.west) -- (\ux,-1);

\node[navybar!80,fill=white,inner sep=1.5pt,anchor=east]
  (vlabel) at ({-\K-0.7},{\vy+\BarTop+0.8})
  {$\mathbf v=(-k,3k-2)$};
\draw[->,navybar!80,line width=0.7pt]
  (vlabel.east) -- (-\K,\vy);

\end{tikzpicture}
\caption{The periodic vertical-bar construction, illustrated for $k=4$.
The navy points form the model bar $B_k$ of $2k-2$ selected vertices.  Its
closed exact-distance-$k$ coverage region is the shaded hexagonal lattice
region $C_k$.  The remaining bars are periodic copies; the bottom vertices
repeat under the translations $\mathbf u=(2k+1,-1)$ and
$\mathbf v=(-k,3k-2)$.}
\label{fig:exact-k-vertical-bar-construction}
\end{figure}
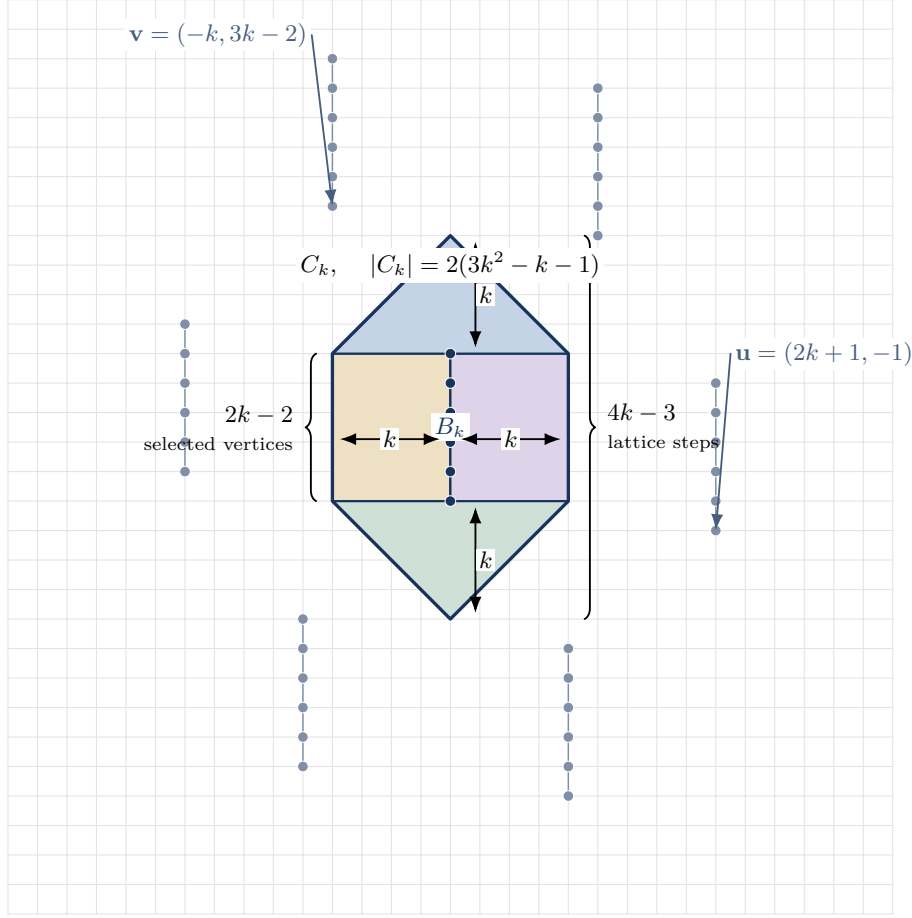

\subsection{Passing to the finite grid}

For the finite grid \(G_n\), define the width-\(k\) boundary strip
\[
\partial_k G_n
=
\bigl\{(i,j)\in V(G_n):
\min\{i,j,n+1-i,n+1-j\}\leq k\bigr\}.
\]
The remaining interior has side length $r=\max\{n-2k,0\}$.
For a shift \(t\in\mathbb{Z}/M\mathbb{Z}\), set
\begin{equation}
D_t
=
\partial_k G_n
\cup
\bigl\{(i,j)\in V(G_n)\setminus\partial_k G_n:
\phi_k(i,j)\in t+R_k\bigr\}.
\label{eq:exact-k-upper-finite-set}
\end{equation}

\begin{lemma}
\label{lem:exact-k-upper-boundary}
For every \(t\in\mathbb Z/M\mathbb Z\), the set
\[
D_t
=
\partial_kG_n
\cup
\bigl\{
(i,j)\in V(G_n)\setminus\partial_kG_n:
\phi_k(i,j)\in t+R_k
\bigr\}
\]
is an exact-distance \(k\)-dominating set of \(G_n\).
\end{lemma}

\begin{proof}
Let \(v=(i,j)\in V(G_n)\setminus D_t\). Since
\(\partial_kG_n\subseteq D_t\), we have $k+1\leq i,j\leq n-k$.
Moreover, \(v\notin D_\infty(t)\), because \(v\) is an interior vertex and
\(\phi_k(v)\notin t+R_k\). Proposition~\ref{prop:exact-k-upper-infinite}
therefore gives a vertex
\[
w=v+(a,b)\in D_\infty(t)
\]
such that $|a|+|b|=k$.
In particular, \(|a|,|b|\leq k\), and hence
\[
1\leq i+a\leq n,
\; \text{and} \;
1\leq j+b\leq n.
\]
Thus \(w\in V(G_n)\).

If \(w\in\partial_kG_n\), then \(w\in D_t\) by definition. Otherwise,
\(w\notin\partial_kG_n\), and since \(w\in D_\infty(t)\), we have
\(\phi_k(w)\in t+R_k\); hence \(w\in D_t\) again. Finally, $d_{G_n}(v,w)=|a|+|b|=k$. This finishes the proof.
\end{proof}

Now we are ready to prove the upper bound estimate.

\begin{theorem}
\label{thm:exact-k-general-upper}
For every \(k\geq 2\) and every positive integer \(n\),
\[
D_{\mathrm{opt}}^{(k)}(G_n)
\leq
n^2-r^2+
\left\lfloor
\frac{(k-1)r^2}{3k^2-k-1}
\right\rfloor,
\qquad
r=\max\{n-2k,0\}.
\]
Consequently, for every fixed \(k\geq 2\), $D_{\mathrm{opt}}^{(k)}(G_n)
\leq
\frac{k-1}{3k^2-k-1}n^2+O_k(n)$.

\end{theorem}

\begin{proof}
By Lemma~\ref{lem:exact-k-upper-boundary}, each of the \(M\) sets \(D_t\)
is feasible. It remains only to choose a shift with few selected interior
vertices.

Fix an interior vertex \(z\). The condition $\phi_k(z)\in t+R_k$
is equivalent to $t\in\phi_k(z)-R_k$.
The residues in \(R_k\) are distinct, because \(R_k\) is the image of the
bar \(B_k\subseteq C_k\) under the bijection of
Lemma~\ref{lem:exact-k-upper-complete-residues}. Therefore \(z\) is selected
for exactly $|R_k|=L=2k-2$
of the \(M=2(3k^2-k-1)\) shifts.

There are \(r^2\) interior vertices. Summing the number of selected interior
vertices over all \(M\) shifts and then dividing by \(M\), their average is $\frac{2k-2}{M}r^2
=
\frac{k-1}{3k^2-k-1}r^2$.

Hence at least one shift selects at most the floor of this average. The
boundary strip contains exactly \(n^2-r^2\) vertices, so for that shift
\[
|D_t|
\leq
n^2-r^2+
\left\lfloor
\frac{(k-1)r^2}{3k^2-k-1}
\right\rfloor.
\]
Finally, \(r=n-O_k(1)\), which gives the asymptotic estimate.
\end{proof}





\section{Conclusion}

We studied exact-distance \(k\)-domination in the square grid
$G_n$. For every fixed \(k\geq3\), we proved that the
exact-distance \(k\)-domination density $\delta_k$ satisfies $\frac{1}{4k}
\leq
\delta_k
\leq
\frac{k-1}{3k^2-k-1}$.

For \(k=2\), the trivial lower bound~(\ref{obs:trivial lower bound}) and upper bound coincide and give $\delta_2=\frac{1}{9}$.
The lower bound follows from unavoidable overlaps among exact-distance
coverage sets, whereas the upper bound is obtained from an explicit
periodic vertical-bar construction.

The principal open problem is to determine \(\delta_k\) exactly, beginning with $\delta_3$ which we believe is exactly equal to $2/23$.
Further directions include improving the overlap argument, finding periodic
constructions of smaller density, extending the problem to higher-dimensional
lattices, and other planar lattices and other interesting graph classes.

\begin{remark}[Rectangular grids]
\label{rem:rectangular-grids}
Note that the square-grid results extend naturally to the rectangular ($m\times n$) grid $G_{m,n}$. Let $a=\max\{(m-2k),0\}, b=\max\{(n-2k),0\}$, and $\alpha_k=\frac{k-1}{3k^2-k-1}$. Then it is easy to argue that
$\left\lceil \delta_k mn\right\rceil
\leq
D_{\mathrm{opt}}^{(k)}(G_{m,n})
\leq
mn-ab+\left\lfloor \alpha_k ab\right\rfloor$ .

The lower bound follows from a replication argument. Indeed, if an
\(m\times n\) rectangle admitted a solution of density smaller than
\(\delta_k\), then a sufficiently large square whose side length is
divisible by both \(m\) and \(n\) could be tiled by copies of this rectangle.
Repeating the same solution independently in every copy would produce square
grids of density smaller than \(\delta_k\), contradicting the definition of
the exact-distance \(k\)-domination density.

For the upper bound, select every vertex in the width-\(k\) boundary strip
of \(G_{m,n}\), leaving an \(a\times b\) interior, and use the same periodic
residue construction as in the square case on this interior. An unselected
interior vertex is at least \(k\) steps from every side, so its
distance-\(k\) witness supplied by the infinite-lattice construction remains
inside the rectangle; if the witness lies in the boundary strip, then it is
already selected. Averaging over all residue shifts as before selects at most an
\(\alpha_k\)-fraction of the \(ab\) interior vertices, which gives the
displayed upper bound.

Consequently, for every fixed \(k\geq3\),
$\left\lceil\frac{mn}{4k}\right\rceil
\leq
D_{\mathrm{opt}}^{(k)}(G_{m,n})
\leq
\frac{k-1}{3k^2-k-1}\,mn+O_k(m+n)$.

\end{remark}

\section{Declaration of Competing Interest}
The authors declare that they have no known competing financial interests or personal relationships that could have appeared to influence the work reported in this paper.

\section{Acknowledgements}
The author Arpan Sadhukhan gratefully acknowledges support from the Anusandhan National Research Foundation (ANRF) through the ANRF-National Post Doctoral Fellowship (N-PDF) for Mathematical Sciences, File No. PDF/2025/001644.

\section{Data availability}
No data was used for the research described in the article.

\bibliography{references}

@article{henning20172,
  title={On 2-step and hop dominating sets in graphs},
  author={Henning, Michael A. and Rad, Nader Jafari},
  journal={Graphs and Combinatorics},
  volume={33},
  number={4},
  pages={913--927},
  year={2017},
  publisher={Springer}
}

@article{ayyaswamy2018note,
  title={A note on hop domination number of some special families of graphs},
  author={Ayyaswamy, SK. and Natarajan, C. and Sathiamoorphy, G.},
  journal={International Journal of Pure and Applied Mathematics},
  volume={119},
  number={12},
  pages={14165--14171},
  year={2018}
}

@article{ayyaswamy2015bounds,
  title={Bounds on the hop domination number of a tree},
  author={Ayyaswamy, SK. and Krishnakumari, B. and Natarajan, C. and Venkatakrishnan, YB.},
  journal={Proceedings-Mathematical Sciences},
  volume={125},
  number={4},
  pages={449--455},
  year={2015},
  publisher={Springer}
}

@article{fujita2025tight,
  title={Tight upper bounds on the hop domination number of triangle-free graphs},
  author={Fujita, Shinya and Park, Boram},
  journal={arXiv preprint arXiv:2503.04124},
  year={2025}
}

@article{shanmugavelan2021hop,
  title={On hop domination number of some generalized graph structures},
  author={Shanmugavelan, SK. and Natarajan, C.},
  journal={Ural Mathematical Journal},
  volume={7},
  number={2 (13)},
  pages={121--135},
  year={2021},
  publisher={Федеральное государственное бюджетное учреждение науки {\guillemotleft}Институт математики~…}
}

@article{anusha2024further,
  title={Further results on the hop domination number of a graph},
  author={Anusha, D. and Robin, S. Joseph and John, J.},
  journal={Boletim Da Sociedade Paranaense de Matem{\'a}tica},
  volume={42},
  pages={1--12},
  year={2024}
}

@article{packiavathihop,
  title={Hop Domination Number of Caterpillar Graphs},
  author={Balamurugan, S. and Gnana Jothi, R. B. and Getchial Pon Packiavathi, P.},
  journal={Advances in Mathematics: Scientific Journal},
  volume={9},
  number={5},
  pages={2739--2748},
  year={2020},
  month={July},
  issn={1857-8365},
  eissn={1857-8438},
  doi={10.37418/amsj.9.5.38}
}

@article{karthika2025hop,
  title={Hop Domination on subclasses of Perfect graphs},
  author={Karthika, D. and Muthucumaraswamy, R. and Bhyravarapu, Sriram and Kumar, Pritesh},
  journal={Theoretical Computer Science},
  pages={115547},
  year={2025},
  publisher={Elsevier}
}

@article{henning2020algorithm,
  title={Algorithm and hardness results on hop domination in graphs},
  author={Henning, Michael A. and Pal, Saikat and Pradhan, Dinabandhu},
  journal={Information Processing Letters},
  volume={153},
  pages={105872},
  year={2020},
  publisher={Elsevier}
}

@inproceedings{karthika2025polynomial,
  title={Polynomial Time Algorithms for Hop Domination},
  author={Karthika, D. and Muthucumaraswamy, R. and Bhyravarapu, Sriram and Kumar, Pritesh},
  booktitle={Conference on Algorithms and Discrete Applied Mathematics},
  pages={85--96},
  year={2025},
  organization={Springer}
}

@book{HaynesHedetniemiSlater1998,
  author    = {Teresa W. Haynes and Stephen T. Hedetniemi and Peter J. Slater},
  title     = {Fundamentals of Domination in Graphs},
  series    = {Monographs and Textbooks in Pure and Applied Mathematics},
  volume    = {208},
  publisher = {Marcel Dekker},
  address   = {New York},
  year      = {1998},
  isbn      = {978-0-8247-0033-1}
}

@article{GoncalvesPinlouRaoThomasse2011,
  author  = {Daniel Gon{\c{c}}alves and Alexandre Pinlou and
             Micha{\"e}l Rao and St{\'e}phan Thomass{\'e}},
  title   = {The Domination Number of Grids},
  journal = {SIAM Journal on Discrete Mathematics},
  volume  = {25},
  number  = {3},
  pages   = {1443--1453},
  year    = {2011},
  doi     = {10.1137/11082574}
}

@inproceedings{FataSmithSundaram2013,
  author    = {Elaheh Fata and Stephen L. Smith and Shreyas Sundaram},
  title     = {Distributed Dominating Sets on Grids},
  booktitle = {Proceedings of the 2013 American Control Conference},
  pages     = {211--216},
  publisher = {IEEE},
  year      = {2013},
  doi       = {10.1109/ACC.2013.6579839}
}

@article{FarinaGrez2016,
  author  = {Michael Farina and Armando Grez},
  title   = {New Upper Bounds on the Distance Domination Numbers of Grids},
  journal = {Rose-Hulman Undergraduate Mathematics Journal},
  volume  = {17},
  number  = {2},
  note    = {Article 7},
  year    = {2016}
}

@article{BlessingInskoJohnsonMauretour2015,
  author  = {David Blessing and Erik Insko and Katie Johnson and
             Christie Mauretour},
  title   = {On \((t,r)\) Broadcast Domination Numbers of Grids},
  journal = {Discrete Applied Mathematics},
  volume  = {187},
  pages   = {19--40},
  year    = {2015},
  doi     = {10.1016/j.dam.2015.02.005}
}

@article{NatarajanAyyaswamy2015,
  author  = {C. Natarajan and SK.  Ayyaswamy},
  title   = {Hop Domination in Graphs---II},
  journal = {Analele Stiintifice ale Universitatii Ovidius Constanta,
             Seria Matematica},
  volume  = {23},
  number  = {2},
  pages   = {187--199},
  year    = {2015},
  doi     = {10.1515/auom-2015-0036}
}

@article{ChartrandHararyHossainSchultz1995,
  author  = {Gary Chartrand and Frank Harary and Moazzem Hossain and
             Kelly Schultz},
  title   = {Exact 2-Step Domination in Graphs},
  journal = {Mathematica Bohemica},
  volume  = {120},
  number  = {2},
  pages   = {125--134},
  year    = {1995},
  doi     = {10.21136/MB.1995.126228}
}

@article{Hersh1999,
  author  = {Patricia Hersh},
  title   = {On Exact \(n\)-Step Domination},
  journal = {Discrete Mathematics},
  volume  = {205},
  number  = {1--3},
  pages   = {235--239},
  year    = {1999},
  doi     = {10.1016/S0012-365X(99)00024-2}
}

@article{jalalvand2017complexity,
  author    = {Farhadi Jalalvand, M. and Jafari Rad, N.},
  title     = {On the complexity of $k$-step and $k$-hop dominating sets in graphs},
  journal   = {Mathematica Montisnigri},
  volume    = {40},
  pages     = {36--41},
  year      = {2017}
}

\newpage

\end{document}